\documentclass{amsart}
\usepackage[utf8]{inputenc}
\usepackage{amsmath}
\usepackage{amsthm}
\usepackage{amsfonts, dsfont}
\usepackage{amssymb}
\usepackage[all]{xy}
\usepackage{indentfirst}
\usepackage[pagebackref=true]{hyperref}

\usepackage{cleveref}
\usepackage[alphabetic, initials]{amsrefs}

\usepackage{faktor}
\usepackage{xfrac}

\newtheorem{thm}{Theorem}[section]

\newtheorem{theorem}[thm]{Theorem}
\newtheorem{corollary}[thm]{Corollary}
\newtheorem{proposition}[thm]{Proposition}

\newtheorem{lemma}[thm]{Lemma}
\newtheorem*{theorem*}{Theorem}
\newtheorem*{corollary*}{Corollary}

\theoremstyle{definition}

\newtheorem*{defn*}{Definiton}

\newtheorem*{ack}{Acknowledgements}
\newtheorem{remark}[thm]{Remark}

\newcommand{\N}{\mathbb{N}} 
\newcommand{\Z}{\mathbb{Z}} 

\usepackage{verbatim} 
\usepackage{color}
\newcommand{\G}{\Gamma}

\DeclareMathOperator{\Ima}{Im}

\newcommand{\La}{\Lambda}
\newcommand{\la}{\lambda}
\newcommand{\Id}{\mathrm{Id}}

\newcommand{\Fix}{\operatorname{Fix}}

\newcommand{\supp}{\operatorname{supp}}

\newcommand{\sS}{\mathsf{S}}

\newcommand{\cL}{\mathcal{L}}

\newcommand{\cP}{\mathcal{P}}

\title{The AH conjecture for Cantor minimal dihedral systems}
\author{Eduardo Scarparo}

\address{Eduardo Scarparo\\ Center for Engineering\\ Federal University of Pelotas\\ Brazil}
\email{eduardo.scarparo@ufpel.edu.br}
\thanks{This project has received funding from the European Research Council (ERC) under the European Union's Horizon 2020 research and innovation programme (grant agreement No. 817597).}
\begin{document}
\begin{abstract}
The AH conjecture relates the low-dimensional homology groups of a groupoid with the abelianization of its topological full group. We show that transformation groupoids of minimal actions of the infinite dihedral group on the Cantor set satisfy this conjecture.  The proof uses Kakutani–Rokhlin partitions adapted to such systems.
\end{abstract}
\maketitle
\section{Introduction}
The AH conjecture of Matui in \cite{Mat16} proposes that, given an effective minimal étale groupoid $G$ with unit space homeomorphic to the Cantor set, there is an exact sequence

$$H_0(G)\otimes\Z_2\stackrel{j}\longrightarrow [[G]]_{\mathrm{ab}}\stackrel{I}\longrightarrow H_1(G)\longrightarrow 0,$$
where $H_*(G)$, for $*=0,1$, are homology groups of $G$, $[[G]]$ is the topological full group of $G$, $I\colon[[G
]]_{\mathrm{ab}}\to H_1(G)$ is the \emph{index map}  and $j\colon H_0(G)\otimes\Z_2\to[[G]]_{\mathrm{ab}}$ is a map introduced by Nekrashevych in \cite{Nek19}. This conjecture has been verified for several classes of examples, e.g., principal almost finite groupoids (\cite{Mat16}), transformation groupoids of odometers (\cite{Sca20}), graph groupoids (\cite{Mat15} and \cite{NO21}) and Katsura–Exel–Pardo groupoids (\cite{NyOr21}).

In \cite{Mat16},  Matui also made a conjecture (called the \emph{HK conjecture}) relating homology groups of a groupoid $G$ and the K-theory of the reduced $C^*$-algebra of $G$. In \cite{Sca20} a counterexample was given to this conjecture  involving a certain non-free minimal action of the infinite dihedral group $\Z\rtimes\Z_2$ on the Cantor set (further counterexamples were later provided in \cite{OrSc22}, \cite{OS22} and \cite{D22}).

Minimal actions of $\Z\rtimes\Z_2$ on the Cantor set (what we call \emph{Cantor minimal dihedral systems}) have been studied in several different contexts.  For example, in \cite{BEK93},  Bratteli, Evans and Kishimoto showed that, if the action is not free, then the associated crossed product is approximately finite-dimensional. In \cite{Tho10}, Thomsen showed that the converse of this statement holds. In \cite[Example 3]{BRY18}, Baake,  Roberts and Yassawi gave an example of an aperiodic minimal subshift whose associated $\Z$-action cannot be extended to an action of $\Z\rtimes\Z_2$. In \cite{Jia21}, Jiang showed that continuously orbit-equivalent minimal actions of $\Z\rtimes\Z_2$ are conjugate.

In \cite{OrSc22}, Ortega and the author showed that (transformation groupoids of) Cantor minimal dihedral systems satisfy the HK conjecture if and only if the action is free. Furthermore, it was shown that such systems are always almost finite. By Matui's result (in \cite{Mat16}), it follows that, in the case that the action is free, the AH conjecture holds. But the question of whether non-free actions also satisfy it was left open.

In this work, we address this problem. The main tool is a technique of Grigorchuk and Medynets (\cite{GM14}) which, given a minimal homeomorphism on the Cantor set,  provides a decomposition of the elements of the topological full group as products of elements adapted to Kakutani–Rokhlin partitions.  

We also pose two questions about the range of homological invariants of Cantor minimal dihedral systems that we hope will be of interest to the dynamical systems and operator algebras communities (see Remark \ref{rem:range}).

\begin{remark}
In a recent preprint (\cite[Corollary E]{Li22}), Li proved that minimal, ample groupoids with comparison satisfy a stronger version of the AH conjecture.  By \cite[Lemma 6.7]{Mat12} and \cite[Theorem 2.10]{OrSc22}, transformation groupoids of Cantor minimal dihedral systems have comparison, thus are covered by Li's result. Li's methods are much deeper, however,  and involve tools from algebraic K-theory.
\end{remark}

\begin{ack}
I thank the referee for their many corrections and useful comments which helped to greatly improve the presentation of this work.
\end{ack}

\section{Preliminaries}
\subsection{Homology of transformation groupoids}

In this section, we provide an ad hoc definition of the low-dimensional homology groups of a transformation groupoid. See, for example, \cite[Section 3.1]{Mat12} for a more general perspective.

Let $X$ and $Y$ be totally disconnected locally compact Hausdorff spaces. Given a local homeomorphism $\pi\colon X\to Y$,  there is a map $\pi_*\colon C_c(X,\Z)\to C_c(Y,\Z)$ given by $$\pi_*(f)(y):=\sum_{x\in\pi^{-1}(y)}f(x),$$ for $f\in C_c(X,\Z)$ and $y\in Y$. Notice that, given $U\subset X$ compact-open such that $\pi|_U$ is a homeomorphism,  we have that $\pi_*(1_U)=1_{\pi(U)}$.

Let $\G$ be a discrete group acting by homeomorphisms on $X$.  As a space, the \emph{transformation groupoid} of the action,  denoted by $\G\ltimes X$, is the set $\G\times X$ endowed with the product topology.  The \emph{range} and \emph{source} maps $r,s\colon\G\ltimes X\to X$ are given by $r(g,x):=gx$ and $s(g,x):=x$, for $(g,x)\in\G\ltimes X$.  The product of two elements $(h,y)$ and $(g,x)$ is defined if and only if $y=gx$, in which case $(h,gx)(g,x):=(hg,x)$. Inversion is given by $(g,x)^{-1}:=(g^{-1},gx)$.  Let $(\G\ltimes X)^{(2)}:=\{(a_1,a_2)\in(\G\ltimes X)^2:s(a_1)=r(a_2)\}.$

A \emph{bisection} is a compact open subset $U\subset\G\ltimes X$ such that $r|_U$ and $s|_U$ are homeomorphisms.  In this case, there exist $g_1,\dots,g_n\in\G$ and clopen sets $A_1,\dots,A_n\subset X$ such that $s(U)=\bigsqcup A_i$, $r(U)=\bigsqcup g_i A_i$ and $U=\bigsqcup\{g_i\}\times A_i$. 

Given bisections $U$ and $V$, the set 

\begin{equation*}
UV:=\{a_1a_2:(a_1,a_2)\in (U\times V)\cap(\G\ltimes X)^{(2)}\}
\end{equation*}
is a bisection as well, and so is $U^{-1}=\{a^{-1}:a\in U\}$.

Let $d_0,d_1,d_2\colon (\G\ltimes X)^{(2)}\to\G\ltimes X$ be the local homeomorphisms given by

\begin{align*}
d_i(a_1,a_2):=
\begin{cases}
a_2 & \text{if $i=0$}\\
a_1a_2 & \text{if $i=1$}\\
a_1 & \text{if $i=2$}. 
\end{cases}
\end{align*} 

Define $\delta_2\colon C_c((\G \ltimes X)^{(2)},\Z)\to C_c(\G\ltimes X,\Z)$ by $\delta_2:=d_{0_*}-d_{1_*}+d_{2_*}$ and $\delta_1\colon C_c(\G\ltimes X,\Z)\to C_c(X,\Z)$ by $\delta_1:=s_*-r_*$.  Notice that  $\delta_1(\delta_2(1_{U\times V}))=0$ for all bisections $U,V$ such that $s(U)=r(V)$.  Since $C_c((\G \ltimes X)^{(2)},\Z)$ is spanned by such $1_{U\times V}$, we conclude that $\delta_1\circ\delta_2=0$.  We let $H_0(\G\ltimes X):=\frac{C_c(X,\Z)}{\Ima\delta_1}$ and $H_1(\G\ltimes X):=\frac{\ker\delta_1}{\Ima\delta_2}$.

\subsection{Cantor minimal dihedral systems}

An action of a group $\G$ on a compact Hausdorff space $X$ is said to be \emph{topologically free} if, for any $g\in\G\setminus\{e\}$, the set $\Fix_g:=\{x\in X:gx=x\}$ has empty interior. It is said to be \emph{minimal} if every point has dense orbit.  If $\varphi$ is a homeomorphism on $X$, then we say that $\varphi$ is minimal if the associated $\Z$-action is minimal.

Recall that the \emph{infinite dihedral group} is the group $\Z\rtimes\Z_2$, where the $\Z_2$-action is given by multiplication by $-1$.  Any action $\alpha$ of $\Z\rtimes\Z_2$ on a space $X$ is given by a pair $(\varphi,\sigma)$ of homeomorphisms on $X$ such that $\sigma\varphi\sigma^{-1}=\varphi^{-1}$ and $\sigma^2=\Id_X$.  In this case,  given $(n,i)\in\Z\rtimes\Z_2$, we have that $\alpha_{n,i}=\varphi^n\sigma^i$.  Alternatively, since $\Z\rtimes\Z_2\simeq\Z_2*\Z_2$, any action of the infinite dihedral group on $X$ is given by a pair of homeomorphisms $a$ and $b$ on $X$ such that $a^2=b^2=\Id_X$. Moreover, we can identify $\sigma$ with $a$, and $\varphi\sigma$ with $b$.

\begin{remark}\label{remark}
In \cite[Proposition 2.8]{OrSc22}, it was observed that any minimal action $(\varphi,\sigma)$ of $\Z\rtimes\Z_2$ on the Cantor set is topologically free. Furthermore, if the action is not free, then $\varphi$ is minimal, and $\Fix_\sigma$ and $\Fix_{\varphi\sigma}$ are disjoint sets (because otherwise $\varphi$ would admit a fixed point),  at least one of which is nonempty. 
\end{remark}

Notice that $\sigma$ induces a homomorphism 
\begin{align*}
\sigma_*\colon H_0(\Z\ltimes X)&\to H_0(\Z\ltimes X)\\
[f]&\mapsto[f\circ\sigma],
\end{align*}
for $f\in C_c(X,\Z)$.

We recall the following computation of homology groups of Cantor minimal dihedral systems. In order to ease the notation, we denote $\Id_{H_0(\Z\ltimes X)}$ by $1$.

\begin{theorem}[{\cite[Theorem 3.6]{OrSc22}}]\label{thm:hom}
Let $\alpha:=(\varphi,\sigma)$ be a minimal action of $\Z\rtimes\Z_2$ on the Cantor set $X$.
\begin{enumerate}
\item[(i)] If $\varphi$ is not minimal, then there exists a clopen $\varphi$-invariant set $Y\subset X$ such that $X=Y\sqcup\sigma(Y)$ and $\varphi|_Y$ is minimal. Furthermore,
\begin{align*}
H_0((\Z\rtimes\Z_2)\ltimes X)&\simeq H_0(\Z\ltimes Y),\\
H_1((\Z\rtimes\Z_2)\ltimes X)&\simeq\Z.
\end{align*}
\item[(ii)] If $\varphi$ is minimal and $\alpha$ is free, then
\begin{align*}
H_0((\Z\rtimes\Z_2)\ltimes X)&\simeq \Z_2\oplus (1+\sigma_*)H_0(\Z\ltimes X),\\
H_1((\Z\rtimes\Z_2)\ltimes X)&=0.
\end{align*}
\item[(iii)] If $\alpha$ is not free, then
\begin{align}
H_0((\Z\rtimes\Z_2)\ltimes X)&\simeq (1+\sigma_*)H_0(\Z\ltimes X),\nonumber\\
H_1((\Z\rtimes\Z_2)\ltimes X)&\simeq C(\mathrm{Fix}_\sigma\sqcup \mathrm{Fix}_{\varphi\sigma},\Z_2).\label{nonfree}
\end{align}

\end{enumerate}
\end{theorem}

\begin{remark}\label{rem:range}
(i) We do not know any example of a minimal action $(\varphi,\sigma)$ of $\Z\rtimes\Z_2$ on the Cantor set for which $\Fix_\sigma$ and $\Fix_{\varphi\sigma}$ are not both finite.  For odometers and actions coming from irrational rotations this was shown to be the case in \cite[Lemma 3.2]{Sca20} and \cite[Corollary 4.4]{BEK93}, respectively.  Let $(X,\varphi)$ be a two-sided shift space coming from a primitive substitution. Suppose that, given any word $w_1\dotsb w_n$ in the language $\mathcal{L}(X)$, it holds that $w_n\dotsb w_1\in \cL(X)$ as well.  Then the map $\sigma\colon X\to X$ defined by $\sigma(x)_n:=x_{-n}$ is well-defined and induces an action $(\varphi,\sigma)$ of $\Z\rtimes\Z_2$ on $X$. It follows from \cite[Theorem 1.1]{DZ00} that $\Fix_\sigma$ and $\Fix_{\varphi\sigma}$ are both finite.

(ii) Let $(\varphi,\sigma)$ be an action of $\Z\rtimes\Z_2$ on the Cantor set $X$ such that $\varphi$ is minimal.  For odometers and actions coming from irrational irrations, it was shown in \cite[Proposition 3.3]{Sca20} and \cite[Corollary 4.4]{BEK93}, respectively, that $\sigma_*=\Id_{H_0(\Z\ltimes X)}$. In such cases,  $1+\sigma_*$ becomes simply multiplication by $2$. Since $H_0(\Z\ltimes X)$ is always torsion-free,  it follows that $(1+\sigma_*)H_0(\Z\ltimes X)\simeq H_0(\Z\ltimes X)$. It would be interesting to find examples in which $\sigma_*\neq\Id_{H_0(\Z\ltimes X)}$.

\end{remark}

We will now present an explicit description of the isomorphism in \eqref{nonfree} which will be necessary later.

Let $\G$ be a group acting on a compact Hausdorff space $X$.  Given $f\in C_c(\G\ltimes X,\Z)$ and $w\in \G$,  we let $f_w\in C(X,\Z)$ be given by $f_w(x):=f(w,x)$,  for $x\in X$.  If $g\in \G$ and $x\in X$, then $gf_w(x):=f(w,g^{-1}x)$. 

Given $w\in\Z_2*\Z_2$, let $w_1\dotsb w_n$ be the reduced expression of $w$ in the letters $a,b$ associated to the generators of $\Z_2*\Z_2$.  Denote the \emph{length} of $w$ by $|w|$.  Given $1\leq i\leq j\leq n$, let $w_{(i,j]}:=w_{i+1}\dotsb w_j$ if $i<j$, and $e$ otherwise. 
\begin{lemma}\label{wd}
Let $(a,b)$ be a non-free minimal action of $\Z_2*\Z_2$ on the Cantor set $X$. There is a map $R\colon C_c((\Z_2*\Z_2)\ltimes X,\Z)\to C(\Fix_a\sqcup\Fix_b,\Z_2)$ given by
$$R(f):=\sum_{w\in\Z_2*\Z_2}\sum_{i=1}^{|w|}(w_{(i,|w|]}f_w)1_{\Fix_{w_i}}\mod 2$$
that vanishes on $\Ima\delta_2$.
\end{lemma}
\begin{proof}
We will first show that the image of $R$ consists of continuous functions. Given $A\subset X$ clopen and $w\in\Z_2*\Z_2$, notice that 
$$R(1_{\{w\}\times A})=\sum_{i=1}^{|w|}1_{(w_{(i,|w|]}A)\cap\Fix_{w_i}}$$
(the two sides of this equation are viewed as functions taking values in $\Z_2$). Since $C_c((\Z_2*\Z_2)\ltimes X,\Z)$ is generated by elements of the form $1_{\{w\}\times A}$, it follows that $R$ takes values in $C(\Fix_a\sqcup\Fix_b,\Z_2)$. 

We will now show that $R$ vanishes on $\Ima\delta_2$. Given $u,v\in \Z_2*\Z_2$, there exist $u',x,v'\in\Z_ 2*\Z_2$ such that $u=u'x^{-1}$ and $v=xv'$ are in reduced form as is $u'v'$. Fix $A\subset X$ clopen and let $U:=\{u\}\times vA$ and $V:=\{v\}\times A$.  Then $UV=\{u'v'\}\times A$ and

\begin{equation}\label{image}
\delta_2(1_{(U\times V)\cap((\Z_2*\Z_2)\ltimes X)^{(2)}})=1_U-1_{UV}+1_V.
\end{equation}

Furthermore,  
\begin{align*}
R(1_U-1_{UV}+1_V)
=\sum_{i=1}^{|u|}1_{(u_{(i,|u|]}vA)\cap\Fix_{u_i}}-\sum_{i=1}^{|u'v'|}1_{((u'v')_{(i,|u'v'|]}A)\cap\Fix_{(u'v')_i}}\\
+\sum_{i=1}^{|v|}1_{(v_{(i,|v|]}A)\cap\Fix_{v_i}}
=\sum_{i=1}^{|u'|}1_{(u'_{(i,|u'|]}x^{-1}vA)\cap\Fix_{u'_i}}
+\sum_{i=1}^{|x|}1_{((x^{-1})_{(i,|x|]}vA)\cap\Fix_{(x^{-1})_i}}\\
-\sum_{i=1}^{|u'|}1_{(u'_{(i,|u'|]}v'A)\cap\Fix_{u'_i}}
-\sum_{i=1}^{|v'|}1_{(v'_{(i,|v'|]}A)\cap\Fix_{v'_i}}
+\sum_{i=1}^{|x|}1_{(x_{(i,|x|]}v'A)\cap\Fix_{x_i}}\\
+\sum_{i=1}^{|v'|}1_{(v'_{(i,|v'|]}A)\cap\Fix_{v'_i}}
=\sum_{i=1}^{|x|}1_{((x^{-1})_{(i,|x|]}xv'A)\cap\Fix_{(x^{-1})_i}}
+\sum_{i=1}^{|x|}1_{(x_{(i,|x|]}v'A)\cap\Fix_{x_i}}\\
\stackrel{j:=|x|-i+1}=\sum_{j=1}^{|x|}1_{(x_1\dots x_{j-1})^{-1}xv'A)\cap\Fix_{x_j}}+\sum_{i=1}^{|x|}1_{(x_{(i,|x|]}v'A)\cap\Fix_{x_i}}\\
=\sum_{j=1}^{|x|}1_{(x_jx_{(j,|x|]}v'A)\cap\Fix_{x_j}}+\sum_{i=1}^{|x|}1_{(x_{(i,|x|]}v'A)\cap\Fix_{x_i}}=0,
\end{align*}
where the last equality follows from the fact that, given $B\subset X$ and $g\in\Z_2*\Z_2$, we have $gB\cap\Fix_g=B\cap\Fix_g$.

Since $\Ima\delta_2$ is generated by elements of the form given in \eqref{image},  the result follows.

\end{proof}
The proof of the following result is exactly the same as \cite[Lemma 4.38]{Tho10}, so we omit it.
\begin{lemma}\label{lem:tho}
 Let $U$ be a totally disconnected compact Hausdorff space with $|U|>1$ and $\sigma$ a homeomorphism on $U$ such that $\sigma^2=\Id_U$. If $\Fix_\sigma=\emptyset$, then there exists a clopen set $V\subset U$ such that $U=V\sqcup \sigma(V)$.
\end{lemma}
\begin{theorem}\label{thm:form}
Let $(a,b)$ be a non-free minimal action of $\Z_2*\Z_2$ on the Cantor set $X$. The map $R\colon H_1((\Z_2*\Z_2)\ltimes X)\to C(\Fix_a\sqcup\Fix_b,\Z_2)$ given by
$$R([f]):=\sum_{w\in\Z_2*\Z_2}\sum_{i=1}^{|w|}(w_{(i,|w|]}f_w)1_{\Fix_{w_i}}\mod 2,$$
for $f\in\ker\delta_1$, is an isomorphism.

\end{theorem}
\begin{proof}
The fact that $R$ is well-defined follows from Lemma \ref{wd}.

\emph{Surjectivity:} Given a clopen subspace $U'$ of $\Fix_a$, there exists an open set $U\subset X$ such that $U'=U\cap\Fix_a$.  By the compactness of $U'$, we can cover it by finitely many clopen sets in $X$ each of which is contained in $U$, so $U$ can be taken as the union of these clopen sets and assumed to be clopen in $X$ as well.  Since $a^2=\Id_X$, by considering $U\cap a(U)$ we may also assume that $U=a(U)$. Then $R([1_{\{a\}\times U}])=1_{U\cap\Fix_a}=1_{U'}$.  Surjectivity of $R$ follows by arguing in the same way for $b$.

\emph{Injectivity:} Fix $f\in\ker\delta_1$ such that $R([f])=0$, and we will show that $[f]=0$. Given a clopen set $A\subset X$ and $g,h\in\Z_2*\Z_2$, we have by \eqref{image} that 
\begin{equation}\label{rela}
[1_{\{h\}\times gA}-1_{\{hg\}\times A}+1_{\{g\}\times A}]=0.
\end{equation}
 In particular, $[1_{\{e\}\times A}]=0$.  Furthermore, 
\begin{equation}\label{mais}
1_{\{hg\}\times A}+\Ima\delta_2=1_{\{h\}\times gA}+1_{\{g\}\times A}+\Ima\delta_2.
\end{equation}

Notice that $C_c((\Z_2*\Z_2)\rtimes X,\Z)$ is spanned by $$\{1_{\{w\}\times A}:w\in\Z_2*\Z_2,\text{$A\subset X$ clopen}\}$$ and any $w\in \Z_2*\Z_2$ can be written as a product of the elements $a$ and $b$.  Therefore, by using \eqref{mais} repeatedly, we obtain $f_1,f_2\in C_c((\Z_2*\Z_2)\rtimes X,\Z)$ such that $\supp f_1\subset\{a\}\times X$, $\supp f_2\subset\{b\}\times X$ and $f+\Ima\delta_2=f_1+f_2+\Ima\delta_2$. Since $f\in\ker\delta_1\supset\Ima\delta_2$, we have that $\delta_1(f_1+ f_2)=0$, hence $[f]=[f_1+f_2].$ 

We claim that $\delta_1(f_1)=\delta_1(f_2)=0$. Indeed, we can write $f_1=\sum_{n\in\Z}n1_{\{a\}\times A_n}$ and $f_2=\sum_{n\in\Z}n1_{\{b\}\times B_n}$ for families of disjoint clopen sets $(A_n)_{n\in\Z},(B_n)_{n\in\Z}$.  Since $f_1$ and $f_2$ are compactly supported, only finitely many of the sets $A_n$ and $B_n$ are nonempty.  We have
\begin{align}
0=\delta_1(f_1+f_2) 
&\implies \delta_1(f_1)=-\delta_1(f_2)\nonumber\\
&\implies \sum n(1_{A_n}-1_{aA_n})=\sum n(1_{bB_n}-1_{B_n}).\label{aga}
\end{align}

Let $h:=\sum n(1_{A_n}-1_{aA_n})$.  Since $a^2=b^2=\Id_X$, we have that $h\circ a=-h$ and it is a consequence of \eqref{aga} that $h\circ b=-h$, hence $h\circ (ba)=h$. Since the action is not free,  it follows from Remark \ref{remark} that $ba$ is a minimal homeomorphism. Therefore,  $h$ is a constant function.  Moreover,  by Remark \ref{remark} either $a$ or $b$ has a fixed point $x$. In particular, $h(x)=0$ and $h=0$. Hence, $\delta_1(f_1)=\delta_1(f_2)=0$. 

Let us now show that $[f_1]=0$.  Since $\delta_1(f_1)=0$, we have that $\sum n1_{aA_n}=\sum n1_{A_n}$. Hence $a(A_n)=A_n$ for every $n\in \Z\setminus\{0\}$.  By applying \eqref{rela}, we conclude that for any clopen set $D\subset X$ such that $a(D)=D$,  it holds that $[2(1_{\{a\}\times D})]=0$. Hence, there exists a clopen set $A\subset X$ such that $a(A)=A$ and $[f_1]=[1_{\{a\}\times A}]$.  Since $R([f])=0$,  it follows that $0=R([f_1])=1_{A\cap\Fix_a}$, hence $A\cap\Fix_a=\emptyset$.  By using Lemma \ref{lem:tho},  we obtain a clopen set $A'\subset X$ set such that $A=A'\sqcup a(A')$.

Finally, $[f_1]=[1_{\{a\}\times a(A')}+1_{\{a\}\times A'}]=0$ by \eqref{rela} again. 

By arguing in the same way for $f_2$, we conclude that $[f]=0$.

\end{proof}

\subsection{Topological full groups}
Let $\alpha$ be an action of a group $\G$ on the Cantor set X.  We say that a bisection $U\subset\G\ltimes X$ is \emph{full} if $r(U)=s(U)=X$.  Take $g_1,\dots,g_n\in\G$ and disjoint clopen sets $A_1,\dots,A_n\subset X$ such that $X=\bigsqcup A_i=\bigsqcup g_i A_i$ and $U=\bigsqcup\{g_i\}\times A_i$.  We can associate to $U$ a homeomorphism $\theta_U$ on $X$ given by $\theta_U(x)=g_i x$, for $x\in A_i$.  Given full bisections $U$ and $V$, it holds that $\theta_{UV}=\theta_U\theta_V$ and $\theta_{U^{-1}}=(\theta_U)^{-1}$. The \emph{topological full group} of $\alpha$ is $[[\alpha]]:=\{\theta_U:\text{$U$ is a full bisection}\}$.

Notice that $[[\alpha]]$ coincides with the set of homeomorphisms $h\colon X\to X$ which are locally given by $\alpha$, in the sense that there are $g_1,\dots,g_n\in\G$ and clopen sets $A_1,\dots,A_n\subset X$ such that $X=\bigsqcup A_i=\bigsqcup g_i A_i$ and $h|_{A_i}=\alpha_{g_i}|_{A_i}$ for $1\leq i \leq n$.

From now on,  suppose $\G$ is countable and $\alpha$ is minimal and topologically free.  Topological freeness implies that $\theta$ is injective (in particular, $[[\alpha]]$ coincides with the group of full bisections, which is usually taken as the definition of the topological full group of the transformation groupoid of $\alpha$).  The \emph{index map} $I\colon[[\alpha]]_{\mathrm{ab}}\to H_1(\G\ltimes X)$ is given by $I([\theta_U]):=[1_U]$ (\cite[Definition 7.1]{Mat12}).

Given $g\in\G$ and a clopen set $A\subset X$ such that $A\cap gA=\emptyset$, let $\tau_{g,A}\in[[\alpha]]$ be the map given by

\begin{align}\label{tau}
\tau_{g,A}(x):=
\begin{cases}
gx& \text{if $x\in A$}\\
g^{-1}x & \text{if $x\in gA$}\\
x & \text{otherwise}. 
\end{cases}
\end{align}

Since the clopen sets $A, gA, (A\sqcup gA)^c$ partition $X$, it follows that $\tau_{g,A}$ is a homeomorphism. 

Let $\sS(\alpha)$ be the subgroup of $[[\alpha]]$ generated by such elements.  By \cite[Theorem 7.2]{Nek19}, there exists a unique map $j\colon H_0(\G\ltimes X)\otimes\Z_2\to[[\alpha]]_{ab}$ which, given $g\in \G$ and $A\subset X$ as above, maps $[1_A]\otimes 1$ to $[\tau_{g,A}]$ (the fact that $j$ does not depend on the choice of $g$ is part of the content of \cite[Theorem 7.2]{Nek19}).  Notice that 
\begin{equation}\label{eq:sa}
\Ima j=\{[h]\in[[\alpha]]_{ab}:h\in \sS(\alpha)\}.
\end{equation}

It follows from the results in \cite{Nek19} that $\Ima j\subset \ker I$,  but it is not difficult to check this directly by using \eqref{rela}.  Matui's AH conjecture (\cite[Conjecture 2.9]{Mat16}) predicts that the following sequence is exact:

 $$H_0(\G\ltimes X)\otimes\Z_2\stackrel{j}\longrightarrow [[\alpha]]_{\mathrm{ab}}\stackrel{I}\longrightarrow H_1(\G\ltimes X)\longrightarrow 0.$$

Therefore,  the AH conjecture consists of two problems: determining whether $I$ is surjective and whether $\ker I \subset \Ima j$.

\section{The AH Conjecture for Cantor minimal dihedral systems}

In this section, we show that minimal actions of $\Z\rtimes\Z_2$ on the Cantor set satisfy the AH conjecture.  We will use the following result from \cite{BEK93} which provides \emph{Kakutani–Rokhlin partitions} for such systems.  All the intervals consist of integers, so that, for example, given an integer $K>0$,  $[1,K]=\{1,\dots,K\}$.

\begin{proposition}[{\cite{BEK93}}]\label{partition}
Let $(\varphi,\sigma)$ be a minimal action of $\Z\rtimes\Z_2$ on the Cantor set $X$ such that $\varphi\sigma$ admits a fixed point. Given a partition $\cP$ of $X$ into clopen sets and $N\in\N$, there exist clopen sets $Y_1,\dots,Y_K\subset X$ and integers $J_1,\dots,J_K$ such that the following hold.
\begin{enumerate}
\item[(i)] For $k\in[1,K]$, $J_k\geq 2N$.
\item[(ii)] The sets $\{\varphi^i(Y_k):k\in[1,K],i\in[0,J_k)\}$ are disjoint with union $X$. Moreover, this partition refines $\cP$.
\item[(iii)] Let $Y:=\bigsqcup_{k=1}^KY_k$.  Then $\sigma(Y)=\varphi^{-1}(Y)=\bigsqcup_{k=1}^K\varphi^{J_k-1}(Y_k)$ and, for $i\in[-N,N)$,  there exists $P_i\in\cP$ such that $\varphi^i(Y)\subset P_i$.

\item[(iv)] There exists an involutive bijection $\tau\colon[1,K]\to[1,K]$ such that, for $l\in\Z$ and $k\in[1,K]$, we have $J_{k}=J_{\tau(k)}$ and $\sigma\varphi^l(Y_k)=\varphi^{J_k-l-1}(Y_{\tau(k)})$.
\end{enumerate}
\end{proposition}
\begin{proof}

Items (i), (ii) and (iii) are the content of \cite[Proposition 1.2 and Equation 1.13]{BEK93}.  Let us prove (iv). By \cite[Remark 1.3]{BEK93},  given $k\in[1,K]$, there exists $\tau(k)\in[1,K]$ such that $\sigma(Y_k)=\varphi^{J_{\tau(k)}-1}(Y_{\tau(k)})$. Furthermore, $J_k=J_{\tau(k)}$ and  $\sigma(Y_{\tau(k)})=\varphi^{J_{k}-1}(Y_{k})$, hence $\tau$ is an involution.  

Recall that $\sigma\varphi\sigma^{-1}=\varphi^{-1}$.  Given $l\in \Z$, we have 

$$\sigma\varphi^l(Y_k)=\varphi^{-l}\sigma(Y_k)=\varphi^{J_{k}-l-1}(Y_{\tau(k)}).$$
Analogously,
$$\sigma\varphi^l(Y_{\tau(k)})=\varphi^{-l}\sigma(Y_{\tau(k)})=\varphi^{J_{\tau(k)}-l-1}(Y_k).$$
\end{proof}

Given a Kakutani–Rokhlin partition as in Proposition \ref{partition} and $k\in[1,K]$,  the set $\{\varphi^i(Y_k):i\in[0,J_k)\}$ is called a \emph{tower} of the partition and, for $i\in[0,J_k)$,  each set $\varphi^i(Y_k)$ is called a \emph{level} of the tower.

Given  a minimal action $(\varphi,\sigma)$ of $\Z\rtimes\Z_2$ on the Cantor set $X$,  denote by $[[(\varphi,\sigma)]]$ the associated topological full group.  Let $[[\varphi]]$ be the topological full group associated to the $\Z$-action induced by $\varphi$.  Also let $\sS(\varphi,\sigma)$ and $\sS(\varphi)$ be the groups generated by elements as in \eqref{tau}, for the actions of $\Z\rtimes\Z_2$ and $\Z$, respectively.  

Given $h\in[[(\varphi,\sigma)]]$, there are $(r_1,s_1),\dots,(r_n,s_n)\in\Z\rtimes\Z_2$ and clopen sets $A_1,\dots,A_n\subset X$ such that $X=\bigsqcup A_i=\bigsqcup \varphi^{r_i}\sigma^{s_i}(A_i)$ and $h|_{A_i}=\varphi^{r_i}\sigma^{s_i}|_{A_i}$ for $1\leq i \leq n$. Notice that $h\in[[\varphi]]$ if and only if $s_i=0$ for $1\leq i\leq n$.

Given a clopen set $A\subset X$ and $h\in[[(\varphi,\sigma)]]$ such that $h(A)=A$,  let $h_A\in [[(\varphi,\sigma)]]$ be given by $h_A(x):=h(x)$, for $x\in A$, and $h_A(x):=x$ otherwise.

\begin{lemma}\label{lem:fund}
Let $n\in\Z$ and $U\subset X$ be a clopen set such that $\varphi^n\sigma(U)=U$. If $U\cap\Fix_{\varphi^n\sigma}=\emptyset$, then $(\varphi^n\sigma)_U\in\sS(\varphi,\sigma)$.
\end{lemma}
\begin{proof}
By Lemma \ref{lem:tho},  there exists a clopen set $V\subset U$ such that $U=V\sqcup \varphi^n\sigma(V)$. Then $(\varphi^n\sigma)_U=\tau_{\varphi^n\sigma,V}$, as in \eqref{tau}.
\end{proof}

Let $(\varphi,\sigma)$ be a non-free minimal action of $\Z\rtimes\Z_2$ on the Cantor set $X$. The next result says that, given $Q\in[[(\varphi,\sigma)]]$,  by mutiplying $Q$ by some well-behaved involutions, we can assume that the result lies in $[[\varphi]]$. The idea behind its proof is a technique introduced in \cite[Theorem 4.7]{GM14} for factorizing an element of the topological full group of a Cantor minimal $\Z$-system as a product of elements adapted to a Kakutani–Rokhlin partition of the system.

The strategy of the proof is as follows: given $Q\in[[(\varphi,\sigma)]]$,  we will produce a Kakutani–Rokhlin partition $\{\varphi^i(Y_k):k\in[1,K],i\in[0,J_k)\}$ of $X$ such that $Q|_{\varphi^i(Y_k)}=\varphi^{f_1(k,i)}\sigma^{f_2(k,i)}|_{\varphi^i(Y_k)}$, for a family 
$$\{(f_1(k,i),f_2(k,i))\}_{\substack{k\in[1,K]\\ i\in [0,J_k)}}\subset\Z\rtimes\Z_2.$$
 If $f_2(k,i)$ were equal to $0$ for each $(k,i)$, we would have that $Q\in [[\varphi]]$.  We will find disjoint sets $A_1,\dots,A_l$ consisting of unions of the levels $\varphi^i(Y_k)$ on which $f_2(k,i)=1$, and integers $n_1,\dots,n_l$ such that $\varphi^{n_s}\sigma(A_s)=A_s$ for each $1\leq s\leq l$. By multiplying $Q$ by $\prod_{s=1}^l(\varphi^{n_s}\sigma)_{A_s}$, the result will lie in $[[\varphi]]$.  In order to find the sets $A_1,\dots,A_l$, we will have to deal separately with lower and upper levels of towers on one hand, and middle levels of towers on the other, which explains the appearance of the sets $\La_1$ and $\La_2$ in the proof.

\begin{theorem}\label{tec}
Let $(\varphi,\sigma)$ be a non-free minimal action of $\Z\rtimes\Z_2$ on the Cantor set $X$. Given $Q\in[[(\varphi,\sigma)]]$, there exist $n_1,\dots,n_l\in\Z$ and disjoint clopen sets $U_1,\dots,U_l\subset X$ such that $\varphi^{n_i}\sigma(U_i)=U_i$ for $1\leq i\leq l$, and $Q\prod_{i=1}^l(\varphi^{n_i}\sigma)_{U_i}\in[[\varphi]]$.
\end{theorem}
\begin{proof}
We claim that we can assume that $\varphi\sigma$ has a fixed point. If it does not, then $\sigma$ does and we can consider the  action of $\Z\rtimes\Z_2$ given by $(\psi,\mu):=(\varphi^{-1},\varphi\sigma)$. For $r\in\Z$, we have $\psi^r\mu=\varphi^{-r+1}\sigma$ and, in particular, $\psi\mu=\sigma$ has a fixed point. Moreover,  since 
$$\{\varphi^n\sigma^i:n\in\Z,i\in\{0,1\}\}=\{\psi^m\mu^j:m\in\Z,j\in\{0,1\}\},$$
we have $[[(\varphi,\sigma)]]=[[(\psi,\mu)]]$. Hence we can assume that $\varphi\sigma$ has a fixed point.

Let $A_1,\dots,A_n$ be disjoint clopen sets with union $X$ and let $$(m_1,p_1),\dots,(m_n,p_n)\in\Z\rtimes\Z_2$$ such that, for $1\leq i \leq n$ and $x\in A_i$, we have  $Q(x)=\varphi^{m_i}\sigma^{p_i}(x)$.  Denote by $\cP$ the partition determined by the sets $A_i$. Also let $N:=\max\{|m_i|:1\leq i\leq n\}$.  If $N=0$, let 
$$U:=\bigsqcup_{\{i:p_i=1\}}A_i.$$ Since $Q|_{U^c}=\Id_{U^c},$ $Q|_U=\sigma|_U$ and $Q$ is bijective,  we have that $\sigma(U)=U$, $Q=\sigma_U$, and the result follows.  So let us assume that $N>0$.

Adopt the notation of Proposition \ref{partition}.  Given $k\in[1,K]$ and $i\in[0,J_k)$, let $f(k,i):=(f_1(k,i),f_2(k,i))\in\Z\rtimes\Z_2$ be given by $f(k,i):=(m_j,p_j)$ if $\varphi^i(Y_k)\subset A_j$.  Notice that, for any $k,l\in[1,K]$ and $i\in[0,N)$,  by Proposition \ref{partition}.(iii) we have $f(k,i)=f(l,i)$ and $f(k,J_k-i-1)=f(l,J_l-i-1)$.

Given $k\in[1,K]$, let $S_k:=\{i\in[N,J_k-N):f_2(k,i)=1\}$. We will show that $|S_k|=|S_{\tau(k)}|$.  Obviously, we may assume $k\neq\tau(k)$. 

Let $T_k:=\{f_1(\tau(k),n)+J_k-n-1:n\in S_{\tau(k)}\}$.  Given $n\in S_{\tau(k)}$,  by Proposition \ref{partition}(iv) we have $Q(\varphi^n(Y_{\tau(k)}))=\varphi^{f_1(\tau(k),n)+J_k-n-1}(Y_k)$. Furthermore, since $N\leq n<J_k-N$ and $-N\leq f_1(\tau(k),n)\leq N$, we have $$0\leq f_1(\tau(k),n)+J_k-n-1<J_k.$$ 
Since $Q$ is bijective, we obtain that, given $n,m\in S_{\tau(k)}$ such that $n\neq m$, it holds that $f_1(\tau(k),n)+J_k-n-1\neq f_1(\tau(k),m)+J_k-m-1$. Therefore, $|T_k|=|S_{\tau(k)}|$. 

Identify $[0,J_k)$ with the group $\Z_{J_k}$ and let $F\colon\Z_{J_k}\setminus S_k\to\Z_{J_k}$ be the function given by

\begin{align*}
F(t):=
\begin{cases}
f_1(k,t)+t& \text{if $f_2(k,t)=0$}\\
f_1(k,t)-t-1 & \text{if $f_2(k,t)=1.$}
\end{cases}
\end{align*} 

We will show that $\Ima F\supset \Z_{J_k}\setminus T_k$.  Given $m\in\Z_{J_k}\setminus T_k,$ let $x\in X$ such that $Q(x)\in\varphi^m(Y_k)$.  Take $l\in[1,K]$ and $i\in[0,J_l)$ such that $x\in\varphi^i(Y_l)$.

Case $l=k$: If $f_2(k,i)=1$, then $i\notin[N,J_k-N)$. Therefore,  $i\notin S_k$ and $F(i)=m$. If $f_2(k,i)=0$, then $i\notin S_k$ and $F(i)=m$.

Case $l=\tau(k)$: Since $m\notin T_k$, we have $i\notin[N,J_k-N)$, hence $f(k,i)=f(\tau(k),i)$ and $F(i)=m$.

Case $l\notin\{k,\tau(k)\}$: Notice that $i\notin[N,J_l-N)$. Suppose first $i\in[0,N)$. If $f_2(l,i)=0$, then $f_1(l,i)+i<0$ and $f_1(l,i)+i=m\mod J_k$, hence $F(i)=m$. If $f_2(l,i)=1$, then $f_1(l,i)+J_l-i-1\geq J_l$ and $f_1(l,i)-i-1=m$. In particular, $F(i)=m$. 

Finally, assume $i\in[J_l-N,J_l)$.  In particular,  $J_k-J_l+i\in[J_k-N,J_k)$ and $f(l,i)=f(k,J_k-J_l+i)$. If $f_2(l,i)=0$, we have $f_1(l,i)+i-J_l=m$ and $F(J_k-J_l+i)=f_1(k,J_k-J_l+i)+J_k-J_l+i=f_1(l,i)+J_k+i-J_l=m \mod J_k.$ If $f_2(l,i)=1$,  we have $f_1(l,i)+J_l-i-1=m\mod J_k$, hence $F(J_k-J_l+i)=f_1(k,J_k-J_l+i)-J_k+J_l-i-1=m\mod J_k$.

This finishes the proof that $\Ima F\supset \Z_{J_k}\setminus T_k$. Hence,  $|S_k|\leq|T_k|=|S_{\tau(k)}|$. Since $k$ was arbitrary, we get $|S_k|=|S_{\tau(k)}|$ for any $k\in [1,K]$.

Let $\La_1:=\{i\in[-N,N):f_2(1,i)=1\}$. Given $\la\in\La_1$,  by Proposition \ref{partition}.(iii), we have
$$\varphi^{2\la+1}\sigma(\varphi^\la(Y))=\varphi^{2\la+1}\varphi^{-\la-1}(Y)=\varphi^{\la}(Y).$$

Take $\La_2\subset[1,K]$ minimal such that $[1,K]=\La_2\cup\tau(\La_2)$.  For each $k\in\La_2$, let $\alpha_k\colon S_k\to S_{\tau(k)}$ be the order-preserving bijection, and, for $i\in S_k$,  let $n_{k,i}:=-J_k+i+\alpha_k(i)+1$ and $U_{k,i}:=\varphi^i(Y_k)\cup\varphi^{\alpha_k(i)}(Y_{\tau(k)})$ (so that, if $\tau(k)=k$, then $n_{k,i}=-J_k+2i+1$ and $U_{k,i}=\varphi^i(Y_k)$).  Notice that, for each $k\in\La_2$ and $i\in S_k$, it holds that $\varphi^{n_{k,i}}\sigma(U_{k,i})=U_{k,i}$. 

Given $\la\in\La_1\cap[0,N)$, we have that $\varphi^\lambda(Y)$ consists of levels whose distance to the bottom of their towers is $\lambda$. Analogously, given $\la\in\La_1\cap[-N,-1]$,  we have that $\varphi^\lambda(Y)$ consists of levels whose distance to the top of their towers is $|\la|-1$.  In particular, the sets $\varphi^\la(Y)$ are disjoint from each other and from the sets $U_{k,i}$.

Finally, $$Q\prod_{\la\in\La_1}(\varphi^{2\la+1}\sigma)_{\varphi^\la(Y)}\prod_{\substack{k\in\La_2\\i\in S_k}}(\varphi^{n_{k,i}}\sigma)_{U_{k,i}}\in[[\varphi]].$$

\end{proof}

\begin{corollary}\label{cor:fac}
Let $(\varphi,\sigma)$ be a non-free minimal action of $\Z\rtimes\Z_2$ on the Cantor set $X$. Given $Q\in[[(\varphi,\sigma)]]$, there exist $n_1,\dots,n_l\in\Z$ and  clopen sets $V_1,\dots,V_l\subset X$ such that $\varphi^{n_i}\sigma(V_i)=V_i$ for $1\leq i\leq l$, and $Q\prod_{i=1}^l(\varphi^{n_i}\sigma)_{V_i}\in\sS(\varphi)$.
\end{corollary}
\begin{proof}
Recall that we can identify $H_1(\Z\ltimes X)$ with $\Z$ in such a way that the index map $I\colon [[\varphi]]_{ab}\to\Z$ satisfies $I([\varphi])=1$ (see, e.g.,  \cite[Proposition 3.3.1]{Nyl20}). In particular, by \cite[Lemma 6.3]{Mat12} and \cite[Theorem 3.3.(4)]{Mat16}, given $T\in[[\varphi]]$,  we have that $T\varphi^{-I([T])}\in\sS(\varphi)$. Furthermore,  since $\sigma^2=\Id_X$, we have 
\begin{equation}\label{eq:ind}
\varphi^{-I([T])}=((\varphi\sigma)\sigma)^{-I([T])}=\left((\varphi\sigma)_X\sigma_X\right)^{-I([T])}.
\end{equation}

Given $Q\in[[(\varphi,\sigma)]]$,  by Theorem \ref{tec} there are $n_1,\dots,n_l\in\Z$ and clopen sets $U_1,\dots,U_l\subset X$ such that $T:=Q\prod_{i=1}^l(\varphi^{n_i}\sigma)_{U_i}\in[[\varphi]]$. Hence
\begin{equation}\label{eq:qp}
Q\left(\prod_{i=1}^l(\varphi^{n_i}\sigma)_{U_i}\right)\varphi^{-I([T])}\in\sS(\varphi).
\end{equation}
By \eqref{eq:ind} and \eqref{eq:qp}, the result follows.
\end{proof}

\begin{theorem}\label{thm:main}
Minimal actions of $\Z\rtimes\Z_2$ on the Cantor set satisfy the AH conjecture.
\end{theorem}
\begin{proof}
By \cite[Theorem 2.10]{OrSc22} and \cite[Theorem 3.6]{Mat16},  if the action is free, then it satisfies the AH conjecture.  Furthermore, surjectivity of the index map in general follows from \cite[Theorem 7.5]{Mat12} (alternatively,  this can be easily checked by using Theorem \ref{thm:form}). 

Let $(\varphi,\sigma)$ be a non-free minimal action of $\Z\rtimes\Z_2$ on the Cantor set $X$. Given $Q\in[[(\varphi,\sigma)]]$,  denote the class of $Q$ in $\frac{[[(\varphi,\sigma)]]_{ab}}{\Ima j}$ by $[Q]$. Since $\Ima j\subset \ker I$, we have that the index map factors through $\frac{[[(\varphi,\sigma)]]_{ab}}{\Ima j}$.  Assume that $I([Q])=0$. We want to show that $\ker I\subset \Ima j$,  and will do this by proving that $[Q]$ is the zero element of $\frac{[[(\varphi,\sigma)]]_{ab}}{\Ima j}$.  

It follows from \eqref{eq:sa} and Corollary \ref{cor:fac} that there exist clopen sets $V_1,\dots,V_l$ and $m_1,\dots,m_l\in\Z$ such that $\varphi^{m_i}\sigma(V_i)=V_i$ for $1\leq i\leq l$ and $[Q]=\sum_{i=1}^l[(\varphi^{m_i}\sigma)_{V_i}]$.  

Given a clopen set $A\subset X$ and $n,m\in\Z$ such that $\varphi^m\sigma(A)=A$, notice that $\varphi^{-n}(\varphi^m\sigma)_A\varphi^n=(\varphi^{-2n+m}\sigma)_{\varphi^{-n}(A)}$.  

For $1\leq i\leq l$,  let $n_i$ be the only integer such that $-2n_i+m_i\in\{0,1\}$. If $-2n_i+m_i=0$, let $A_i:=\varphi^{-n_i}(V_i)$ and $B_i:=\emptyset$. If $-2n_i+m_i=1$, let $B_i:=\varphi^{-n_i}(V_i)$ and $A_i:=\emptyset$. Then $A_1,\dots,A_l,B_1,\dots,B_l$ are clopen sets such that $\sigma(A_i)=A_i$ and $\varphi\sigma(B_i)=B_i$ for $1\leq i\leq l$, and $[Q]=\sum_{i=1}^l[(\sigma)_{A_i}]+[(\varphi\sigma)_{B_i}]$. 

Let $C_a:=\sum_{i=1}^l[(\sigma)_{A_i}]$, $C_b:=\sum_{i=1}^l[(\varphi\sigma)_{B_i}]$ and $R\colon H_1((\Z_2*\Z_2)\ltimes X)\to C(\Fix_a\sqcup\Fix_b,\Z_2)$ be the map from Theorem \ref{thm:form} (we identify $a$ with $\sigma$ and $\varphi\sigma$ with $b$). We have 

$$0=R(I([Q]))=R(I(C_a))+R(I(C_b))=\sum_{i=1}^l (1_{A_i\cap\Fix_\sigma}+1_{B_i\cap\Fix_{\varphi\sigma}}).$$
Since $\Fix_\sigma$ and $\Fix_{\varphi\sigma}$ are disjoint and $R$ is an isomorphism,  we conclude that $0=I(C_a)=I(C_b)$.

By taking intersections and symmetric differences of the sets $A_1,\dots, A_l$, we obtain disjoint clopen sets $W_1,\dots,W_n\subset X$ such that $\sigma(W_j)=W_j$ for $1\leq j\leq n$ and such that, for $1\leq i\leq l$, there is $F_i\subset\{1,\dots,l\}$ such that $A_i=\bigsqcup_{j\in F_i}W_j$.  In particular, 
\begin{equation}\label{eq:ca}
C_a=\sum_{i=1}^n\sum_{j\in F_i}[(\sigma)_{W_j}].
\end{equation}
Since each $(\sigma)_{W_j}$ is involutive, we have that $2[(\sigma)_{W_j}]=0$.  By cancelling the terms which appear an even number of times in the sum in \eqref{eq:ca}, we obtain $F\subset\{1,\dots,l\}$ such that $C_a=\sum_{j\in F}[(\sigma)_{W_j}]$.  By Theorem \ref{thm:form}, we have $0=R(I(C_a))=\sum_{j\in F}1_{W_j\cap\Fix_\sigma}$,  hence $W_j\cap\Fix_\sigma=\emptyset$ for each $j\in F$. By using Lemma \ref{lem:fund} and \eqref{eq:sa}, we conclude that $C_a=0$.  

By arguing in an analogous way for $\sum_{i=1}^k[(\varphi\sigma)_{B_i}]$, we conclude that $C_b=0$ and $[Q]=0$.

\end{proof}
In \cite[Section 5]{GPS99}, it was shown that topological full groups of Cantor minimal $\Z$-systems are \emph{indicable} (that is,  they admit a homomorphism onto $\Z$). For Cantor minimal dihedral systems, we have the following:
\begin{corollary}\label{lin}
Let $(\varphi,\sigma)$ be a minimal action of $\Z\rtimes\Z_2$ on the Cantor set. Then $[[(\varphi,\sigma)]]$ is indicable if and only if $\varphi$ is not minimal.
\end{corollary}
\begin{proof}
If $\varphi$ is not minimal, then it follows from surjectivity of the index map and Theorem \ref{thm:hom} that $[[(\varphi,\sigma)]]$ is indicable.

If $\varphi$ is minimal, then it follows from Theorems \ref{thm:hom} and \ref{thm:main} that $[[(\varphi,\sigma)]]_{ab}$ is an extension of torsion groups, hence is a torsion group as well.  In particular,  $[[(\varphi,\sigma)]]$ is not indicable.
\end{proof}
\bibliography{bibliografia}
\end{document}